\documentclass[12pt,leqno,draft]{article}
\usepackage{amsfonts}
\usepackage{amsmath, amsthm, amsfonts, amssymb, color}
\usepackage{mathrsfs}
\usepackage{color}
\newtheorem{thm}{Theorem}[section]

\newtheorem{lem}[thm]{Lemma}

\newtheorem{rem}[thm]{Remark}
\theoremstyle{definition}

\newcommand{\scr}[1]{\mathscr #1}
\definecolor{wco}{rgb}{0.5,0.2,0.3}

\numberwithin{equation}{section} \theoremstyle{remark}

\newcommand{\ua}{\uparrow}

\def\al{\alpha}
\def\be{\beta}
\def\ga{\gamma}
\def\de{\delta}

\def\la{\lambda}

\def\si{\sigma}

\def\til{\tilde}

\def\Ga{\Gamma}

\def\De{\Delta}

\def\Om{\Omega}

\def\Q{\mathbb Q}
\def\R{\mathbb R}
\def\P{\mathbb P}
\def\E{\mathbb E}
\def\H{\mathbb H}
\def\N{\mathbb N}
\def\V{\mathbb V}
\def\M{\mathbb M}

\def\sF{\mathscr F}

\def\sL{\mathscr L}

\def\cF{\mathcal F}

\def\e{\operatorname{e}}

\def\d{\mathrm{d}}

\def\ff{\frac}

\def\ua{\uparrow}
\def\na{\nabla}
\def\pp{\partial}
\def\<{\langle}
\def\>{\rangle}
\def\sq{\sqrt}

\def\1{\mathds{1}}

\def\gap{\text{\rm{gap}}}

\title{{\bf   Estimate of Heat Kernel for Euler-Maruyama Scheme of SDEs Driven by $\alpha$-Stable Noise and Applications} \footnote{Supported in
 part by  NNSFC (11801406).} }
\author{
{\bf  Xing Huang$^{a)}$,  Yongqiang Suo$^{*b)}$, Chenggui Yuan$^{c)}$}\\
\footnotesize{$^{a)}$Center for Applied Mathematics, Tianjin
University, Tianjin 300072, China}\\
\footnotesize{$^{b)}$School of Mathematics, Nanjing University of Aeronautics and Astronautics,
	Nanjing 211106,  China}\\
\footnotesize{$^{c)}$Department of Mathematics, Swansea University,
Bay Campus, Swansea, SA1 8EN, UK}\\
\footnotesize{ \sf{xinghuang@tju.edu.cn},\ \sf{suoyongqiang@nuaa.edu.cn},\ \sf{c.yuan@swansea.ac.uk }}\\
}
\begin{document}
\allowdisplaybreaks
\def\R{\mathbb R}  \def\ff{\frac} \def\ss{\sqrt} \def\B{\mathbf
B}
\def\N{\mathbb N} \def\kk{\kappa} \def\m{{\bf m}}
\def\ee{\varepsilon}\def\ddd{D^*}
\def\dd{\delta} \def\DD{\Delta} \def\vv{\varepsilon} \def\rr{\rho}
\def\<{\langle} \def\>{\rangle} \def\GG{\Gamma} \def\gg{\gamma}
  \def\nn{\nabla} \def\pp{\partial} \def\E{\mathbb E}
\def\d{\text{\rm{d}}} \def\bb{\beta} \def\aa{\alpha} \def\D{\scr D}
  \def\si{\sigma} \def\ess{\text{\rm{ess}}}
\def\beg{\begin} \def\beq{\begin{equation}}  \def\F{\scr F}
\def\Ric{\text{\rm{Ric}}} \def\Hess{\text{\rm{Hess}}}
\def\e{\text{\rm{e}}} \def\ua{\underline a} \def\OO{\Omega}  \def\oo{\omega}
 \def\tt{\tilde} \def\Ric{\text{\rm{Ric}}}
\def\cut{\text{\rm{cut}}} \def\P{\mathbb P} \def\ifn{I_n(f^{\bigotimes n})}
\def\C{\scr C}   \def\G{\scr G}   \def\aaa{\mathbf{r}}     \def\r{r}
\def\gap{\text{\rm{gap}}} \def\prr{\pi_{{\bf m},\varrho}}  \def\r{\mathbf r}
\def\Z{\mathbb Z} \def\vrr{\varrho} \def\ll{\lambda}
\def\L{\scr L}\def\Tt{\tt} \def\TT{\tt}\def\II{\mathbb I}
\def\i{{\rm in}}\def\Sect{{\rm Sect}}  \def\H{\mathbb H}
\def\M{\scr M}\def\Q{\mathbb Q} \def\texto{\text{o}} \def\LL{\Lambda}
\def\Rank{{\rm Rank}} \def\B{\scr B} \def\i{{\rm i}} \def\HR{\hat{\R}^d}
\def\to{\rightarrow}\def\l{\ell}\def\iint{\int}
\def\EE{\scr E}\def\no{\nonumber}
\def\A{\scr A}\def\V{\mathbb V}\def\osc{{\rm osc}}
\def\BB{\scr B}\def\Ent{{\rm Ent}}
\def\U{\scr U}\def\8{\infty} \def\si{\sigma}

\renewcommand{\bar}{\overline}
\renewcommand{\tilde}{\widetilde}
\maketitle

\begin{abstract}
In this paper, the discrete  parametrix method is adopted to investigate the estimation of heat kernel for Euler-Maruyama scheme of SDEs driven by $\alpha$-stable noise, which implies Krylov's estimate and Khasminskii's estimate. As an application, the convergence rate of Euler-Maruyama scheme for a class of multidimensional SDEs with singular drift (in aid of Zvonkin's transformation) is obtained.
\end{abstract} \noindent
 AMS subject Classification:\  60H10, 34K26, 39B72.   \\
\noindent
 Keywords: Zvonkin's transformation, Euler-Maruyama scheme, Heat kernel, Krylov's estimate.
 \vskip 2cm

\section{Introduction}
We consider the following $\R^d$-valued stochastic differential equation (SDE for short)
\begin{align}\label{general model}
X_t=x+\int_0^tb(X_{s})\d s+\int_0^tf(X_{s-})\d L_s,
\end{align}
where $b:\R^d\rightarrow\R^d, f:\R^d\rightarrow\R^{d\times d}$ are measurable functions, and $(L_t)_{t\ge0}$ is a general $\R^d$-valued L\'evy process defined on a complete filtration probability space $(\Om,\sF,(\sF_t)_{t\ge0},\P)$.

If coefficients $b$ and $f$ are Lipschitz continuous, the existence and uniqueness of strong solution to \eqref{general model} is established by the Picard iteration. Moreover, SDE \eqref{general model} can also be numerically solved with the Euler-Maruyama (EM for short) scheme, see \cite{KP} and references therein.

  When the coefficients $b$ and $f$ are irregular, there is a great interest in investigating pathwise uniqueness for SDE \eqref{general model} in the past decades. A useful method in this direction is Zvonkin's transformation which was introduced in \cite{AZ}. This method has been applied to various SDEs, see e.g. \cite{FGP,GM,HW181,KR,XZ,Z,Z1,Z2} and references therein. Furthermore, Zvonkin's transformation has also been applied to investigate the convergence rate of EM scheme for SDEs, we refer to \cite{BHY,DKS,GLN,HK,LS,LS2,LS3} for more details. There are also other transformation methods applied to EM scheme, see \cite{KS,LS,LS2,LS3}.

In the continuous case, $L_s=as+\si W_s$, where $(W_s)_{\{s\ge0\}}$ is a Brownin motion. \cite{LS} establishes the existence and uniqueness result and numerical scheme for \eqref{general model} in the one-dimensional case, the drift therein is piecewise Lipschitz. Their proof is based on a transformation, which  globally transform the piecewise Lipschitz drifts into Lipschitz ones. \cite{LS2,LS3} present a transformation for the multidimensional case which allows to prove an existence and uniqueness result for $d$-dimensional SDEs with discontinuous drift and degenerate diffusion coefficients. Compared with the Zvonkin's transformation, the transformation in \cite{LS2,LS3} does not have to solve a system of parabolic partial differential equations in each step. Recently, the strong convergence rate of EM scheme for SDEs  with integrable drift is obtained by the first author with his co-authors in \cite{BHZ},  and the proof is based on ``parametrix method", which was introduced to obtain existence and estimate on the fundamental solutions of PDEs, see \cite{KM,LM}.

For a general L\'evy process, \cite{HK1} shows that the EM scheme converges strongly with convergence rate $\ff{1}{2}$ for \eqref{general model} with additive noise and  a one-sided Lipschitz continuous drift $b$. \cite{PT} establishes the convergence rate for SDEs with H\"older continuous coefficients driven by Brownian motion and by truncated $\alpha$-stable  processes with index $\alpha>1$.
Moreover, \cite{KS} studies the strong convergence of the EM scheme for a large class of SDEs driven by L\'evy processes such as isotropic $\alpha$-stable, relative stable, layered stable  processes, whose proofs rely on the so-called It\^o-Tanaka trick which relates to the time average $\int_0^tb(X_s)\d s$. Under the assumption that the drift is H\"older continuous, the EM scheme for stochastic functional differential equations with $\alpha$-stable noise is shown in \cite{HL}. It is worth noting that the drifts of SDEs in these literature are assumed to be H\"older(-Dini) continuous or piecewise Lipschitz continuous.
 However, essential difficulty comes up when the drifts only belong to some Sobolev space. More precisely, it is not easy to obtain the estimate like $$\E\left|\int_0^T b(X^{(\delta)}_t)-b(X^{(\delta)}_{t_\delta})\d t\right|^q\leq \Phi(\delta)$$
for some function $\Phi:[0,\infty)\to[0,\infty)$ with $\lim_{\delta\to 0}\Phi(\delta)=0$, here, $X^{(\delta)}_t$ stands for the solution to the numerical SDEs and $b$ is the singular drift, $t_\dd:=\lfloor t/\dd\rfloor\dd$, and $\lfloor t/\dd\rfloor$
denotes the integer part of $t/\dd$.

In this work, we consider the case where 
$(L_t)_{t\ge0}$ is a $d$-dimensional Brownian motion  subordinated by a subordinator, i.e. 
\beq\label{1.1}
\d X_{t}=b(X_t)\d
t+\d W_{S_t}, ~~~t\ge0,\ \ X_{0}=x.
\end{equation}
Herein, $b:\R^d\to\R^d$, $(W_{S_t})_{t\ge0}$ is a rotationally invariant $d$-dimensional $\al$-stable process, with the L\'evy measure $\nu(\d z)=\frac{c_\alpha}{|z|^{d+\alpha}}\d z$ for some constant $c_\alpha>0$. Note that $W_{S_t}$ is composed of processes $W_t$ and $S_t$, where
$(W_t)_{t\ge0}$ is a $d$-dimensional Brownian motion on some complete
filtration probability space $(\OO,\F,(\F_t)_{t\ge0},\P)$. $S_t$ is an $\alpha/2$-stable (with $\alpha\in(1,2)$) subordinator independent of $W$. More precisely, $S_t$ is a non-negative and increasing one-dimensional L\'evy process with Laplace transformation $\E \e^{-\gamma S_t}=\e^{-t\gamma^\frac{\alpha}{2}}, \ \ \gamma,t\geq 0$. By the scaling property,  the process $S_t$ has the same law as  $t^{\frac{2}{\alpha}}S_1$.

 Under  assumptions {\bf (A1)-(A2)}  below, we will adopt discrete parametrix method used in \cite{KM} to obtain explicit upper bounds of heat kernel of  discrete-time EM scheme of \eqref{1.1}. As an application, we investigate the strong convergence rate of the EM scheme.

To end this section, we outline the structure of
 the remaining contents as follows: In Section 2 we state our assumptions and main results. Section 3 is devoted to the notation and some preliminaries. In Section  4, we investigate the estimate of heat kernels of 
 \eqref{E1}. The convergence rate of \eqref{E1} is discussed in Section 5. 

\section{Assumptions and Main results}

To state our main result, we first introduce some notation and facts of Sobolev space, which can be found in \cite{T1}.

For $(p,\gamma)\in[1,\8]\times[0,2]$
, let $H_p^\gamma:=(I-\De)^{-\ff{\gamma}{2}}(L^p(\R^d))$ be the usual Bessel potential space with the norm
\beg{align*}
\|f\|_{\gamma,p}:=\|(I-\De)^{\ff{\gamma}{2}}f\|_p\asymp\|f\|_p+\|(-\De)^{\ff{\gamma}{2}}f\|_p,
\end{align*}
where $\|\cdot\|_p$ is the usual $L^p$-norm in $\R^d$, and $(I-\De)^{\ff{\gamma}{2}}f$ and $(-\De)^{\ff{\gamma}{2}}f$ are defined by the Fourier transformation
$$(I-\De)^{\ff{\gamma}{2}}f:=\cF^{-1}((1+|\cdot|^2)^{\ff{\gamma}{2}}\cF f),\ \
(-\De)^{\ff{\gamma}{2}}f:=\cF^{-1}(|\cdot|^\gamma\cF f).$$
For $p=\8,\gamma=1$, we define $H_\8^1$ as the space of Lipschitz functions with norm
$$\|f\|_{1,\8}:=\|f\|_\8+\|\na f\|_\8.$$

In the sequel, we introduce the Sobolev embedding and an important  inequality on the norm of elements in $H_p^\gamma$, which will be used to  construct the Zvonkin transformation and obtain the  priori estimate of solution to the associated elliptic equation.

For $p\in[1,\8]$ and $\gamma\in[0,2]$,
\beg{equation}\label{embed}
\left\{
\begin{aligned}
	&H_p^\gamma\subset L^q,\: \: q\in\left[p,\ff{dp}{d-\gamma p}\right], \: &\gamma p<d; \\
	&H_p^\gamma\subset H_\8^{\gamma -\ff{d}{p}}\subset C_b^{\gamma -\ff{d}{p}}, \: &\gamma p>d,
\end{aligned}
\right.
\end{equation}
where $C_b^\be$ is the usual H\"older space.
Moreover, for $\gamma\in[0,1]$ and $p\in(1,\8]$, there exists a constant $c$ such that for all $f\in H_p^\gamma$,
\begin{align}\label{bb}
\|f(\cdot+z)-f(\cdot)\|_{p}\le c( |z|^\gamma\wedge 1)\|f\|_{\gamma,p}.
\end{align}

Throughout the paper, we impose the following assumptions on the drift $b$.
\beg{enumerate}
\item[{\bf (A1)}] $\|b\|_{\infty}:=\sup_{x\in\R^d}|b(x)|<\infty$.
\item[{\bf (A2)}]  There exist constants $\beta\in(1-\frac{\alpha}{2},1)$ and $p>(\frac{2d}{\alpha}\vee 2)$
such that $b\in H_p^{\beta}$.
\end{enumerate}

Under ${\bf (A2)}$, \eqref{1.1} has a unique strong
solution $(X_t)_{t\ge0}$, (see, for instance, \cite[Theorem 2.4]{XZ}).

The EM scheme corresponding to \eqref{1.1} is defined as follows:
for any $\dd\in(0,1), $
\begin{equation}\label{E1}
\d X^{(\dd)}_t= b( X^{(\dd)}_{t_\dd})\d t+\d W_{S_t}, \ \ t\ge0,~~~X^{(\dd)}_0=X_0
\end{equation}
with $t_\dd:=\lfloor t/\dd\rfloor\dd$, where  $\lfloor t/\dd\rfloor$
denotes the integer part of $t/\dd$.  We emphasize that
$(X^{(\dd)}_{k\dd})_{k\ge0}$ is a homogeneous Markov process. For $t> s$ and $x\in\R^d$,
$p^{(\dd)}(s,x;t,\cdot)$ denotes the transition density of $ X^{(\dd)}_t
$ with the starting point $X_s^{(\dd)}=x$.

Let $p_\alpha(t,x)$ 
 be the density of $W_{S_t}$.
Our first main result gives an explicit upper bound of the transition
kernel $p^{(\delta)}$.
\begin{thm}\label{lem0}
 Under ${\bf (A1)}$,  there exists a constant $C>0$ such that
\begin{equation}\label{A03}
p^{(\dd)}(j\dd,x;t,y)\le Cp_\alpha(t-j\delta,{y-x}),~~~x,y\in\R^d,~~t>
j\dd,~~\dd\in(0,1).
\end{equation}
\end{thm}
As an application of Theorem \ref{lem0}, the rate of strong convergence for EM Scheme \eqref{E1} can be obtained as follows.
\beg{thm}\label{th1} Assume {\bf (A1)}-{\bf (A2)}. Then, for
$\eta\in(0,2)$, there exist constants $C_1,C_2>0$ such that for any $\epsilon\in(0,1)$
\begin{equation}\label{W1}
 \E\Big(\sup_{0\le t\le T}|X_t-X_t^{(\dd)}|^\eta\Big)  \le C_1
2^{C_2(1+\|b\|^{\frac{2\alpha p}{\alpha p-2d}}_{\beta,p} )}\left(
 \delta^{\frac{\eta\beta}{\alpha }}1_{\{2\beta<\alpha\}}+ \E(S_1^{\frac{\alpha}{2}\epsilon})\delta^{\ff{\eta}{2}\epsilon}1_{\{2\beta\geq \alpha\}}\right).
\end{equation}
\end{thm}
\begin{rem}\label{rem1}
\begin{enumerate}
    \item [(1)] Under certain balance condition, and that the drift $b$ is a bounded function and is H\"older continuous with respect to space, time variable, respectively, the rate of strong convergence of EM scheme \eqref{E1} (inhomogenous case) is established in \cite[Corollary 2.6]{KS}. Since Theorem \ref{th1} is also available for $p=\8$, a close inspection of the Sobolev embedding \eqref{embed} reveals that the result in \cite[Corollary 2.6]{KS} is only the special case of our setting for $p=\8$. Moreover, the optimal balance condition $\beta>\ff{2}{\alpha}-1$ for $\alpha\in(1,2)$ in \cite{KS} is stronger than $\be\in(1-\ff{\alpha}{2},1)$ in {\bf (A2)}.
    \item [(2)] Compared with the result in \cite[Corollary 2.6]{KS}, the moment index $\eta$ in Theorem \ref{th1} is allowed to be greater than $\alpha$, which is reasonable since $X_t-X_t^{(\dd)}$ is a bounded process.
    \end{enumerate}
\end{rem}

\section{Notation and Preliminaries}
Through the paper we use the following notation:

The letter $C,c$ with or without subscripts will denote positive constants, whose value may change in different places. We write $f(x)\asymp g(x)$ to mean that there exist positive constants $C_1,C_2$ such that $C_1 g(x)\le f(x)\le C_2 g(x)$, $f(x)\vee g(x)=\max\{f(x), g(x)\}$, and $f(x)\wedge g(x)=\min\{f(x), g(x)\}$.

Let $\eta_t$ be the density of $S_t$ for $t>0$. It follows from  \cite[Lemma 2.1]{BG} that the density of $W_{S_t}$ has the following expression
\begin{align}\label{de}
p_\alpha(t,x)=\int_0^\infty(2\pi s)^{-\frac{d}{2}}\e^{-\frac{|x|^2}{2s}}\eta_t(s)\d s\asymp t(t^{1/\alpha}+|x|)^{-d-\alpha}, ~(t,x)\in(0,\infty)\times\R^d.
\end{align}
The following inequality will be used frequently:
\beg{align}\label{ineq}
(t^{1/\al}+|x+z|)^{-\ga}\le 4^\ga(t^{1/\al}+|x|)^{-\ga}, ~~\ga\ge 0, ~~ |z|\le (2t^{1/\al})\vee(|x|/2).
\end{align}
Moreover, according to \cite[Lemma 2.2]{CZ}, it is clear  that
\begin{align}\label{na}|\nabla^kp_\alpha(t,x)|\leq C t(t^{1/\alpha}+|x|)^{-d-\alpha-k}\leq Ct^{-k/\alpha}p_\alpha(t,x),\ \ k\in\mathbb{N},
\end{align}
where $\na^k$ stands for the $k$th-order gradient with respect to the spatial variable $x$.

According to the Markov property of $W_{S_t}$, we have
\beg{align}\label{mp}
\int_{\R^d}p_{\alpha}(t-r, x'-y)p_{\alpha}(r-s, y-x)\d y=p_{\alpha}(t-s,x'-x).
\end{align}
In addition, for any $p\geq 1$, \eqref{de} implies that there exists a constant $C>0$ such that
\begin{equation}\label{naao}\begin{split}
\|p_\alpha(t,\cdot)\|_{p}&\leq \left(\int_{\R^d}C t^p(t^{1/\alpha}+|x|)^{-pd-p\alpha}\d x\right)^{\frac{1}{p}}\\
&=\left(\int_{\R^d}C t^pt^{-pd/\alpha-p\alpha/\alpha}(1+|x|/t^{1/\alpha})^{-pd-p\alpha}\d x\right)^{\frac{1}{p}}\\
&=\left(\int_{\R^d}C t^pt^{-pd/\alpha-p\alpha/\alpha}t^{d/\alpha}(1+|y|)^{-pd-p\alpha}\d y\right)^{\frac{1}{p}}\leq Ct^{-d/\alpha+d/(\alpha p)}.
\end{split}\end{equation}


Recall that $\eta_t$ is the density of $S_t$. Let
$$\Theta(r)=\frac{\E\left((2\pi S_1)^{-\frac{d}{2}}\e^{\frac{r}{2S_1}}\right)}{\E(2\pi S_1)^{-\frac{d}{2}}}=\frac{\int_0^\infty(2\pi s)^{-\frac{d}{2}}\e^{\frac{r}{2s}}\eta_1(s)\d s}{\int_0^\infty(2\pi s)^{-\frac{d}{2}}\eta_1(s)\d s},\ \ r\geq 0,$$ which is well defined due to
\begin{align*}
&\E\e^{\gamma S_t^{-1}}\leq \exp{\left[\frac{c\gamma}{t^{\frac{2}{\alpha}}}+ \frac{c\gamma^{\frac{\alpha}{2(\alpha-1)}}}{t^{\frac{1}{\alpha-1}}}\right]}<\infty,\ \ \gamma,t>0,
\end{align*}
for some constant $c>0$, see \cite[Proof of Corollary 2.2]{WW}.
It is not difficult to see that the function $\Theta$ is increasing and continuous on $[0,\infty)$ with $\Theta(0)=1$.

Next, we give a lemma which will be used frequently in the sequel sections.
\begin{lem}\label{FKLE} There exists a constant $C>0$ such that for any $x,M\in\R^d$ and $r>0$,
\begin{equation}\label{FK}\begin{split}
p_\alpha(r,x+M)\leq C4^{d+\al}
p_\alpha(r,x)
\Theta(|M|^2r^{-\frac{2}{\alpha}}).
\end{split}
\end{equation}
\end{lem}
\begin{proof} It follows from the  elemental inequality $|a-\bar{a}|^2\geq \frac{1}{2}|a|^2-|\bar{a}|^2, a,\bar{a}\in\R^d$, the FKG inequality and the scaling property of $S_t$ that
\begin{equation}\label{FK00}\begin{split}
p_\alpha(r,x+M)
&=\int_0^\infty(2\pi s)^{-\frac{d}{2}}\e^{-\frac{|x+M|^2}{2s}}\eta_r(s)\d s\\
&\leq \int_0^\infty(2\pi s)^{-\frac{d}{2}}\e^{-\frac{|x |^2}{4s}}\e^{\frac{|M|^2}{2s}}\eta_r(s)\d s\\
&\leq \frac{\int_0^\infty(2\pi s)^{-\frac{d}{2}}\e^{-\frac{|x|^2}{4s}}\eta_r(s)\d s\int_0^\infty(2\pi s)^{-\frac{d}{2}}\e^{\frac{|M|^2}{2s}}\eta_r(s)\d s}{\int_0^\infty(2\pi s)^{-\frac{d}{2}}\eta_r(s)\d s}\\
&\leq p_\alpha\left(r,\frac{x}{\sqrt{2}}\right)\frac{\int_0^\infty(2\pi s)^{-\frac{d}{2}}\e^{\frac{|M|^2}{2s}}\eta_r(s)\d s}{\int_0^\infty(2\pi s)^{-\frac{d}{2}}\eta_r(s)\d s}\\
	&=p_\alpha\left(r,\frac{x}{\sqrt{2}}\right)\frac{\int_0^\infty(2\pi s)^{-\frac{d}{2}}\e^{\frac{|M|^2r^{-\frac{2}{\alpha}}}{2s}}\eta_1(s)\d s}{\int_0^\infty(2\pi s)^{-\frac{d}{2}}\eta_1(s)\d s}\\
&=p_\alpha\left(r,\frac{x}{\sqrt{2}}\right)\Theta(|M|^2r^{-\frac{2}{\alpha}}).
\end{split}
\end{equation}
This together with \eqref{de} and \eqref{ineq} for $z=\ff{1-\sq{2}}{\sq{2}}x$ and $\ga=d+\al$ yields \eqref{FK}.
\end{proof}

\section{Heat Kernel of EM scheme  \eqref{E1}}
In this section, we first express the heat kernel of solution to discrete-time EM scheme in terms of a sum of convolutional terms with iterated kernels $H^{(\de),(k)}$ and the density of the frozen homogeneous scheme defined \eqref{dis-sch} below, and reveal its explicit upper bounds.
Following this result, we then finish the proof of Theorem \ref{lem0}.

For  $x\in\R^d$  and $j\ge0$, let us begin with the ``frozen" homogeneous  scheme $ (\tt
X^{(\de),j,x,x'}_{i\dd})_{i\ge j}$, which is defined by
\begin{equation}\label{dis-sch}
\tt
X^{(\de),j,x,x'}_{(i+1)\dd}=\tt
X^{(\de),j,x,x'}_{i\dd}+b(x')\de+(W_{S_{(i+1)\dd}}-W_{S_{i\dd}}),~~~i\ge
j,~~~\tt
X^{(\de),j,x,x'}_{j\de}=x.
\end{equation}
Note that, the drift $b$ is frozen at $x'$ in the above definition. In what follows, $p^{(\dd)}(j\dd,x;j'\dd,\cdot)$ and
 $\tt p^{(\dd),x'}(j\dd,x;j'\dd,\cdot)$ denote the transition densities between times $j\dd$ and $j'\dd$ of the discretization scheme \eqref{E1} and the above ``frozen" scheme, respectively.

To derive the kernel of the discrete parametrix representation, we introduce  discrete and homogeneous infinitesimal generators as follows:

For $\psi\in C^2(\R^d;\R)$ and $j\ge0$, we define the family of operators $\mathscr{L}_{j\dd}^{ (\dd) }$ and $\hat{\mathscr{L}}_{j\dd}^{ (\dd) }$ by
\begin{align*}
&(\mathscr{L}_{j\dd}^{ (\dd) }\psi)(x):=
\dd^{-1}\left\{\E(\psi(X_{(j+1)\dd}^{(\dd)})|X_{j\dd}^{(\dd)}=x)-\psi(x)
\right\},\\
& (\hat{\mathscr{L}}_{j\dd}^{ (\dd) }\psi)(x):=\dd^{-1}\left\{\E
\psi(\tilde{X}^{(\de),j,x,x'}_{(j+1)\dd}) -\psi(x)\right\},
\end{align*}
and the discrete kernel $H^{(\dd)}$ as:
\begin{equation}\label{Hd}
\begin{split}
H^{(\dd)}(j\dd,x;j'\dd,x'):&=(\mathscr{L}_{j\dd}^{ (\dd)
}-\hat{\mathscr{L}}_{j\dd}^{ (\dd) })\tt
p^{(\dd),x'}((j+1)\dd,\cdot;j'\dd,x')(x),~j'\ge j+1,
\end{split}
\end{equation}
here we use the convention $\til{p}^{(\de),x'}((j+1)\delta,\cdot;(j+1)\dd, x')=\de_{\{x'\}}(\cdot)$, where $\de_{\{x'\}}(\cdot)$ is the Delta function at the point $x'$.
In what follows, let $0\le j<j'\le\lfloor T/\dd\rfloor.$ According
to the parametrix method in \cite[Proposition 4.1]{LM},  the transition density of EM scheme \eqref{E1} and the transition density of the frozen scheme  \eqref{dis-sch} have the following relationships.
\begin{equation}\label{C6}
p^{(\dd)}(j\dd,x;j'\dd,x')=\sum_{k=0}^{j'-j}(\tt
p^{(\dd),x'}\otimes_\dd H^{(\dd),(k)})(j\dd,x;j'\dd,x'),
\end{equation}
where $\tt p^{(\dd),x'}\otimes_\dd H^{(\dd),(0)}=\tt p^{(\dd),x'},$
$\tt p^{(\dd),x'}\otimes_\dd H^{(\dd),(k)} =(\tt p^{(\dd),x'}\otimes_\dd H^{(\dd),(k-1)})\otimes_\dd H^{(\dd)}$ with
$\otimes_\dd$ being the convolution type binary operation defined by
\begin{equation*}
(f\otimes_\dd
g)(j\dd,x;j'\dd,x')=\dd\sum_{k=j}^{j'-1}\int_{\R^d}f(j\dd,x;k\dd,z)g(k\dd,z;j'\dd,x')\d z.
\end{equation*}

The following lemma gives the smoothing properties of the discrete convolution kernel and the estimate of $p^{(\dd)}(j\de,x;j'\de,x')$.
\begin{lem}\label{THA} Assume {\bf(A1)}. Then there exists a constant $\hat{C}_T>0$ independent of $\delta$ such that for any $0\le j<j'\le \lfloor
T/\dd\rfloor$,
\beg{equation}\begin{split}\label{upp-co}
&|(\tt
p^{(\dd),x'}\otimes_\dd H^{(\dd),(m)})|(j\dd,x;j'\dd,x')\\
&\le
\hat{C}_T^m\|b\|_\8^m((j'-j)\de)^{m(1-\ff{1}{\al})}\ff{\Ga(1-\ff{1}{\al})^m}{\Ga(1+m(1-\ff{1}{\al}))}p_\al((j'-j)\de,x'-x),~~~~~ m\ge0,
\end{split}\end{equation}
where $\Ga(x):=\int_0^\8t^{x-1}\e^{-t}\d t$  is the gamma function. Consequently, it holds 
\begin{equation}\label{C1}
p^{(\dd)}(j\dd,x;j'\dd,x')\le \sum_{m=0}^{j'-j}\ff{[\hat{C}_T\|b\|_\8T^{(1-\ff{1}{\al})}\Ga(1-\ff{1}{\al})]^m}{\Ga(1+m(1-\ff{1}{\al}))}p_\al((j'-j)\de,x'-x).
\end{equation}
\end{lem}

\begin{proof} We divide the proof into two steps.

{\bf Step 1.} We claim that
\begin{equation}\label{A11}
|H^{(\dd)}|(j\dd,x;j'\dd,x')\le
\hat C_T\|b\|_\8((j'-j)\dd)^{\frac{-1}{\alpha}}p_{\alpha}((j'-j)\dd, x'-x),\ \ j'>j.
\end{equation}
Firstly, we prove \eqref{A11} for $j'=j+1.$ It follows from \eqref{de}, \eqref{na}, \eqref{FK} and  \eqref{Hd} that
\begin{equation*}
\begin{split}
&|H^{(\dd)}|(j\dd,x;(j+1)\dd,x')\\
&=\left|(\mathscr{L}_{j\dd}^{ (\dd)
}-\hat{\mathscr{L}}_{j\dd}^{ (\dd) })\tt
p^{(\dd),x'}((j+1)\dd,\cdot;(j+1)\dd,x')(x)\right|\\
&=\left|\ff{1}{\de}\Big\{\E\left(\dd_{\{x'\}}(X_{(j+1)\de}^{(\de)})|X_{j\de}^{(\de)}=x\right)-
\E\left(\dd_{\{x'\}}(\til{X}_{(j+1)\de}^{(\de),j,x,x'})|\til{X}_{j\de}^{(\de),j,x,x'}=x\right)\Big\}\right|\\
&=\ff{1}{\dd}|p^{(\dd)}-\tt
p^{(\dd),x'}|(j\dd,x;(j+1)\dd,x') \\
&=\ff{1}{\dd}\left|p_\alpha(\delta,x'-x-b(x)\delta)-p_\alpha(\delta,x'-x-b(x')\de)\right|\\
&\leq 2\|b\|_\infty\left|\int_0^1\nabla p_\alpha(\delta,x'-x-b(x')\de+\theta (b(x')-b(x))\delta)\d \theta\right|\\
&\leq C\|b\|_\infty\delta^{-\frac{1}{\alpha}}\sup_{\theta\in[0,1]} p_\alpha(\delta,x'-x-b(x')\de+\theta (b(x')-b(x))\delta)\\
&\leq C\|b\|_\infty\delta^{-\frac{1}{\alpha}} p_\alpha(\delta,{x'-x})\Theta(9\delta^{2-\frac{2}{\alpha}}\|b\|^2_\infty)\\
&=:\hat{C}_T\delta^{-\frac{1}{\alpha}}p_\alpha(\delta,{x'-x}).
\end{split}
\end{equation*}
Thus, \eqref{A11} holds for $j'=j+1$.

Next, we are going to
show that \eqref{A11} holds for $j'>j+1.$
According
to  \eqref{de}, \eqref{na}, \eqref{mp}, \eqref{FK} and \eqref{Hd}, it holds that
\begin{align*}
 &|H^{(\dd)}|(j\dd,x;j'\dd,x')\\
&=\ff{1}{\dd}\Big|\Big\{ \int_{\R^d}p_\al(\de,z)\tt
p^{(\dd),x'}((j+1)\dd,x+b(x)\de+z;j'\dd,x')\d z-\tt
p^{(\dd),x'}((j+1)\dd,x;j'\dd,x')\Big\}\\
&\qquad -\Big\{\int_{\R^d}p_\al(\de,z)\tt
p^{(\dd),x'}((j+1)\dd,x+b(x')\de+z;j'\dd,x')\d z-\tt
p^{(\dd),x'}((j+1)\dd,x;j'\dd,x')\Big\}\Big|\\
&=\ff{1}{\dd}\Big| \int_{\R^d}p_\al(\de,z)\Big\{\tt
p^{(\dd),x'}((j+1)\dd,x+b(x)\de+z;j'\dd,x')\\
&\qquad-\tt
p^{(\dd),x'}((j+1)\dd,x+b(x')\de+z;j'\dd,x')\Big\}\d z\Big|\\
&=\ff{1}{\de}\Big|\int_{\R^d}p_\al(\de,z)\Big\{p_\al\left((j'-(j+1))\de,x'-x-b(x)\de-z-b(x')(j'-(j+1))\delta\right)
\\
&\qquad-p_\al\left((j'-(j+1))\de,x'-x-b(x')\de-z-b(x')(j'-(j+1))\delta\right)\Big\}\d z\Big|\\
&=\ff{1}{\dd} \Big|
p_\alpha((j'-j)\dd,x'-x-b(x)\delta-b(x')(j'-(j+1))\delta)\\
&\qquad-
p_\alpha((j'-j)\delta,x'-x-b(x')\delta-b(x')(j'-(j+1))\delta)\Big|\\
&\leq 2\|b\|_\infty\sup_{\theta\in[0,1]}\left|\nabla p_\alpha
((j'-j)\delta,x'-x-b(x')\delta-b(x')(j'-(j+1))\delta+\theta (b(x')-b(x))\delta)\right|\\
&\leq C\|b\|_\infty((j'-j)\delta)^{-\frac{1}{\alpha}}\sup_{\theta\in[0,1]} p_\alpha((j'-j)\delta,x'-x-b(x')(j'-j-1)\de+\theta (b(x')-b(x))\delta) \\
&\le C\|b\|_\infty((j'-j)\delta)^{-\frac{1}{\alpha}} p_\alpha((j'-j)\delta,x'-x)
\Theta(9[(j'-j-1)\delta]^{2-\frac{2}{\alpha}}\|b\|^2_\infty) \\
&=\hat{C}_T\|b\|_\infty((j'-j)\delta)^{-\frac{1}{\alpha}} p_\alpha((j'-j)\delta,x'-x).
\end{align*}

{\bf Step 2.} We are going to prove  \eqref{upp-co}.

Due to \eqref{FK}, it is not difficult to see that
\beg{equation}\label{upper-p}
\begin{split}
\tilde{p}^{(\de),x'}(j\de,x;j'\de,x')&=p_\al((j'-j)\de,x'-x-b(x')(j'-j)\de)\\
&\le C4^{d+\al} \Theta(\|b\|_\8^2((j'-j)\de)^{2-\ff{2}{\al}})p_\al((j'-j)\de,x'-x) \\
&=:\hat{C}_T p_\al((j'-j)\de,x'-x).
\end{split}
\end{equation}
Combining this with \eqref{A11}, we obtain from the definition of operator $\otimes_\de$ that
\beg{align*}
&|\tilde{p}^{(\de),x'}\otimes_\de H^{(\dd)}| (j\dd, x; j'\dd, x')\\
&= \de\sum_{k=j}^{j'-1}\int_{\R^d}\tilde{p}^{(\de),x'}(j\de,x;k\de,z) H^{(\dd)}(k\dd,z;j'\dd,x')\d z\\
&\le\de\sum_{k=j}^{j'-1}\int_{\R^d}\hat{C}_T p_\al((k-j)\de,z-x)\hat{C}_T\|b\|_\8((j'-k)\dd)^{\frac{-1}{\alpha}}p_{\alpha}((j'-k)\dd, x'-z)\d z\\
&=\de \hat{C}_T\|b\|_\8p_\al((j'-j)\dd, x'-x)\sum_{k=j}^{j'-1}((j'-k)\dd)^{\frac{-1}{\alpha}}\\
&\le\hat{C}_T\|b\|_\8p_\al((j'-j)\dd, x'-x)\int_{j\de}^{j'\de}(j'\dd-v)^{\frac{-1}{\alpha}}\d v\\
&=\hat{C}_T\|b\|_\8((j'-j)\de)^{1-\ff{1}{\al}}\be(1,1-\ff{1}{\al})p_\al((j'-j)\dd, x'-x).
\end{align*}
In the above equation $\beta(m,n):=\int_0^ts^{m-1}(1-s)^{n-1}\d s$ stands for the $beta$ function.

Using this and \eqref{A11},  we get
\beg{align*}
&|\tilde{p}^{(\de),x'}\otimes_\de H^{(\dd),(2)}|(j\dd,x;j'\dd,x')\\
&\leq \de\sum_{k=j}^{j'-1}\int_{\R^d}|\tilde{p}^{(\de),x'}\otimes_\de H^{(\dd)}(j\dd,x;k\dd,z)|| H^{(\dd)}(k\dd,z;j'\dd,x')|\d z\\
&\le \de\sum_{k=j}^{j'-1}\int_{\R^d}
\hat{C}_T^2\|b\|_\8^2((k-j)\de)^{1-\ff{1}{\al}}\be(1,1-\ff{1}{\al})p_\al((k-j)\dd, z-x)\\
&\qquad\qquad\qquad\times ((j'-k)\dd)^{\frac{-1}{\alpha}}p_{\alpha}((j'-k)\dd, x'-z)\d z\\
&=\de \hat{C}_T^2\|b\|_\8^2\be(1,1-\ff{1}{\al})\sum_{k=j}^{j'-1}((k-j)\de)^{1-\ff{1}{\al}}((j'-k)\dd)^{\frac{-1}{\alpha}}p_\al((j'-j)\de,x'-x)\\
&\le\hat{C}_T^2\|b\|_\8^2\be(1,1-\ff{1}{\al})
\int_{j\de}^{j'\de}(v-j\de)^{1-\ff{1}{\al}}(j'\dd-v)^{\frac{-1}{\alpha}}\d vp_\al((j'-j)\de,x'-x)\\
&=\hat{C}_T^2\|b\|_\8^2((j'-j)\de)^{2(1-\ff{1}{\al})}\be(1,1-\ff{1}{\al})
\be(2-\ff{1}{\al},1-\ff{1}{\al})p_\al((j'-j)\de,x'-x).
\end{align*}
By an induction argument, one has
\beg{align*}
&|\tilde{p}^{(\de),x'}\otimes_\de H^{(\dd),(m)}|(j\dd,j'\dd,x,x')\\
&\le\hat{C}_T^m\|b\|_\8^m((j'-j)\de)^{m(1-\ff{1}{\al})}p_\al((j'-j)\de,x'-x)\prod_{i=1}^m\be(i-\ff{i-1}{\al}, 1-\ff{1}{\al})\\
&=\hat{C}_T^m\|b\|_\8^m((j'-j)\de)^{m(1-\ff{1}{\al})}p_\al((j'-j)\de,x'-x)\prod_{i=1}^m\ff{\Ga(i-\ff{i-1}{\al})\Ga(1-\ff{1}{\al})}{\Ga(i+1-\ff{i}{\al})}\\
&=\hat{C}_T^m\|b\|_\8^m((j'-j)\de)^{m(1-\ff{1}{\al})}\ff{\Ga^m(1-\ff{1}{\al})}{\Ga(1+m(1-\ff{1}{\al}))}p_\al((j'-j)\de,x'-x).
\end{align*}
Therefore, \eqref{upp-co} is proved, and \eqref{C1} follows from \eqref{C6} and \eqref{upp-co}.
\end{proof}
We are now  in the position to prove Theorem \ref{lem0}.
\begin{proof}[Proof of Theorem \ref{lem0}]
For fixed $t>0$, there is an integer $k\ge0$ such that
$t\in[k\dd,(k+1)\dd).$ It follows from \eqref{FK} that
\begin{equation}\label{A02}
\begin{split}
p^{(\dd)}(k\dd,x;t,y)
&= p_\alpha(t-k\delta,y-x-b(x)(t-k\delta))\\
&\le C4^{d+\al} \Theta((t-k\delta)^{2-\frac{2}{\alpha}}\|b\|^2_\infty)p_\alpha(t-k\delta,{y-x})\\
&=C_1p_\alpha(t-k\delta,{y-x}).
\end{split}
\end{equation}
Note that \eqref{C1} implies
\begin{equation}\begin{split}\label{A02'}
p^{(\dd)}(j\dd,x;j'\dd,x')&\leq \sum_{m=0}^{j'-j}\ff{[\hat{C}_T\|b\|_\8T^{(1-\ff{1}{\al})}\Ga(1-\ff{1}{\al})]^m}{\Ga(1+m(1-\ff{1}{\al}))}p_\al((j'-j)\de,x'-x)\\
&\le C_2 p_{\alpha}((j'-j)\dd, x'-x),~~~j'>j,~x,x'\in\R^d.
\end{split}\end{equation}
Combining this with \eqref{A02} and the Chapman-Kolmogrov equation, we obtain
\begin{equation*}
\begin{split}
p^{(\dd)}(j\dd,x;t,y)&=\int_{\R^d}p^{(\dd)}(j\dd,x;\lfloor
t/\dd\rfloor\dd,z)p^{(\dd)}(\lfloor t/\dd\rfloor\dd,z;t,y)\d z\\
&\le C_1C_2\int_{\R^d}p_\alpha(t-\lfloor t/\dd\rfloor\delta,{y-z})p_{\alpha}((\lfloor t/\dd\rfloor-j)\dd, z-x)\d z\\
&=Cp_\al(t-j\de,y-x).
\end{split}
\end{equation*}
The proof is therefore completed.
\end{proof}
\section{Proof of Theorem \ref{th1}}
Before finishing the proof of Theorem \ref{th1}, we prepare some auxiliary lemmas.
The first lemma below plays a crucial role in the proof of Theorem \ref{th1}.
\beg{lem}\label{lem1}
Assume {\bf (A1)}-{\bf (A2)}, and let $T>0$ be fixed. Then
there exists a constant  $C_T>0 $ such that for any $\epsilon\in(0,1)$,
\begin{equation} \label{b-b}
\int_0^T\mathbb{E}|b(X^{(\dd)}_t)-b(X^{(\dd)}_{t_\dd})|^2\d t\leq C_T\left(
 \delta^{\frac{2\beta}{\alpha }}1_{\{2\beta<\alpha\}}+ \E(S_1^{\frac{\alpha}{2}\epsilon})\delta^{\epsilon}1_{\{2\beta\geq \alpha\}}\right).
 \end{equation}
\end{lem}

\begin{proof}
Observe  that
\begin{equation*}
\begin{split}
\int_0^T\mathbb{E}|b(X^{(\dd)}_t)-b(X^{(\dd)}_{t_\dd})|^2\d
t&=\int_0^\dd\mathbb{E}|b(X^{(\dd)}_t)-b(X^{(\dd)}_0)|^2\d
t\\
&\quad+\sum_{k=1}^{\lfloor
T/\dd\rfloor}\int_{k\dd}^{T\wedge(k+1)\dd}\mathbb{E}|b(X^{(\dd)}_t)-b(X^{(\dd)}_{k\dd})|^2\d
t.
\end{split}
\end{equation*}
It follows from {\bf (A1)} that
\begin{equation}\label{F2}
\int_0^\dd\mathbb{E}|b(X^{(\dd)}_t)-b(X^{(\dd)}_0)|^2\d t\le
4\|b\|_\8^2\dd.
\end{equation}
For  $t\in[k\dd,(k+1)\dd)$,  using the independence
between $X^{(\dd)}_{k\dd}$ and $W_{S_t}-W_{S_{k\dd}}$, and
applying
 Theorem \ref{lem0}, we derive that
\begin{equation}\label{D8}
\begin{split}
&\mathbb{E}|b(X^{(\dd)}_t)-b(X^{(\dd)}_{k\dd})|^2\\
&=\E|b(X^{(\dd)}_{k\dd}+b(X^{(\dd)}_{k\dd})(t-k\dd)+(W_{S_t}-W_{S_{k\dd}}))-b(X^{(\dd)}_{k\dd})|^2\\
&=\int_{\R^d}\int_{\R^d}|b(y+z)-b(y)|^2 p^{(\dd)}(0,x;k\dd,y)p^{(\de)}(k\delta,y;t,z+y)\d y\d z\\
&\le C\int_{\R^d}\int_{\R^d}|b(y+z)-b(y)|^2p_{\al}(k\de,y-x)p_\al(t-k\de,z)\d y\d z.
\end{split}
\end{equation}
By {\bf(A1)}, {\bf(A2)}, H\"older's inequality, \eqref{naao} and \eqref{bb},  we obtain
\beg{align*}
&\int_{\R^d}|b(y+z)-b(y)|^2p_{\al}(k\de,y-x)\d y\\
&\le\Big\{\Big(\int_{\R^d}|b(y+z)-b(y)|^p\d y\Big)^{\ff{2}{p}}\Big(\int_{\R^d}p_\al(k\de,y-x)^{\ff{p}{p-2}}\d y\Big)^{\ff{p-2}{p}}\Big\}
\\
&\le C(k\delta)^{-2d/(\alpha p)}\|b\|_{\be,p}^2(|z|^{2\be}\wedge 1).
\end{align*}
We therefore infer from \eqref{D8} that for any $\epsilon\in(0,1)$ and $t\in[k\delta,(k+1)\delta)$,
\beg{equation}\label{D9}
\begin{split}
&\mathbb{E}|b(X^{(\dd)}_t)-b(X^{(\dd)}_{k\dd})|^2\\
&\le C(k\delta)^{-2d/(\alpha p)}\|b\|_{\be,p}^2\int_{\R^d}\{|z|^{2\be}\wedge 1\}p_\al(t-k\de,z)\{1_{\{2\be\ge\al\}}+1_{\{2\be<\al\}}\}\d z\\
&\leq C(k\delta)^{-2d/(\alpha p)}\|b\|_{\be,p}^2\Big(\E(|W_{S_\delta}|^{2\beta}\wedge 1)1_{\{2\be\ge\al\}}+\de^{\ff{2\be}{\al}}1_{\{2\be<\al\}}\Big)\\
&\leq C(k\delta)^{-2d/(\alpha p)}\|b\|_{\be,p}^2\Big(\E(|W_{S_\delta}|^{\alpha\epsilon})1_{\{2\be\ge\al\}}+\de^{\ff{2\be}{\al}}1_{\{2\be<\al\}}\Big)\\
&= C(k\delta)^{-2d/(\alpha p)}\|b\|_{\be,p}^2\Big(\E|W_1|^{\alpha\epsilon}\E(S_1^{\frac{\alpha}{2}\epsilon})\delta^{\epsilon}1_{\{2\be\ge\al\}}+\de^{\ff{2\be}{\al}}1_{\{2\be<\al\}}\Big),
\end{split}
\end{equation}
where the second inequality is due to the fact that for $2\be<\al$,
\beg{align*}
\int_{\R^d}\{|z|^{2\be}\wedge 1\}p_\al(t,z)\d z
&\asymp\int_{\R^d}\{|z|^{2\be}\wedge 1\}t(t^{\ff{1}{\al}}+|z|)^{-d-\al}\d z\\
&\le C\int_0^\8t\ff{r^{2\be+d-1}}{(t^{\ff{1}{\al}}+r)^{d+\al}}\d r\\
&=\Big(\int_0^{t^{\ff{1}{\al}}}+\int_{t^{\ff{1}{\al}}}^\8\Big)\ff{tr^{2\be+d-1}}{(t^{\ff{1}{\al}}+r)^{d+\al}}\d r\le Ct^{\ff{2\be}{\al}}
.
\end{align*}
Noting that $\int_{0}^{T} r^{-2d/(\alpha p)}\d r<\infty$ due to {\bf(A2)},
we arrive at
\beg{align*}
\sum_{k=1}^{\lfloor
T/\dd\rfloor}\int_{k\dd}^{T\wedge(k+1)\dd}\E|b(X^{(\dd)}_t)-b(X^{(\dd)}_{k\dd})|^2\d t\le C_T\left(\delta^{\frac{2\beta}{\alpha }}1_{\{2\beta<\alpha\}}+ \E|W_1|^{\alpha\epsilon}\E(S_1^{\frac{\alpha}{2}\epsilon})\delta^{\epsilon}1_{\{2\beta\geq \alpha\}}\right).
\end{align*}
This combined with \eqref{F2} implies \eqref{b-b}.
\end{proof}


 Next, we use Theorem \ref{lem0}
to derive the Krylov estimate and  the Khasminskii estimate of  $(X_t^{(\dd)})_{t\ge0}$, see
\cite{GM,KR,XZ,Z,Z2} for more results about Krylov's estimate and Khasminskii's estimate.

\beg{lem}\label{Kry}
Assume {\bf (A1)}. Then,  for
 any $q>(d/\alpha\vee1)$,  there exist constants $C,c>0$ such that Krylov's estimate
 \begin{equation}\label{OO}
\E\Big(\int_s^t|f(X_r^{(\dd)})|\d r\Big|\F_s\Big)\le
C\,\|f\|_{q} (t-s)^{1-d/(\alpha q)},~~~f\in L^q(\R^d), 0\le s\le t\le T,
\end{equation}
holds,
which implies the Khasminskii  estimate
\begin{equation}\label{D1}
\E\exp\left(\ll\int_0^T|f(X^{(\dd)}_t)| \d t\right)\le 2^{1+  T
(c\ll\|f\|_{ q})^{\frac{1}{1-d/(\alpha q)}}},~~~~f\in L^q(\R^d), \ll>0.
\end{equation}
\end{lem}

\begin{proof}
For $0\le s\le t\le T$, note that
\begin{equation*}
\begin{split}
\E\Big(\int_s^t|f(X_r^{(\dd)})|\d
r\Big|\F_s\Big)&=\E\Big(\int_s^{t\wedge(s_\dd+\dd)}|f(X_r^{(\dd)})|\d
r\Big|\F_s\Big)+\E\Big(\int_{t\wedge(
s_\dd+\dd)}^t|f(X_r^{(\dd)})|\d r\Big|\F_s\Big)\\
&=:I_1(s,t)+I_2(s,t).
\end{split}
\end{equation*}
	For $t\in[s,s_\dd+\dd]$,
\begin{equation*}
X_r^{(\dd)}=X_{s_\dd}^{(\dd)}+b(X_{s_\dd}^{(\dd)})(r-s_\dd)+(W_{S_s}-W_{S_{s_\dd}})+(W_{S_r}-W_{S_s}),~~r\in[s,s_\dd+\dd),
\end{equation*}
in view of the independence between $W_{S_r}-W_{S_s}$ and $\sF_s$,
we derive from \eqref{naao} and H\"{o}lder's inequality that for any $q>d/\alpha$,
\begin{equation}\label{H2}
\begin{split}
I_1(s,t)
&=\int_s^{t\wedge(s_\dd+\dd)}\int_{\R^d}f(x+b(x)(r-s_\de)+w+z)p_\alpha(r-s, z)\d
z\Big|_{x=X_{s_\dd}^{(\dd)}}^{w=W_s-W_{s_\dd}}\d r\\
&\le\|f\|_{q}\int_s^{t\wedge(s_\dd+\dd)}(r-s)^{-d/\alpha+d(q-1)/(\alpha q)}\d
r\le \frac{(t-s)^{1-d/(\alpha q)}}{1-d/(\alpha q)}\|f\|_{q}.
\end{split}
\end{equation}
For $t>s_\dd+\dd$, let
$X_{k\dd,r}^{(\dd),x}$ be the EM scheme determined by \eqref{E1}
with $X_{k\dd,k\dd}^{(\dd),x}=x$.  According to the Markov property, it is not difficult to see that
\begin{equation*}
\begin{split}
I_2(s,t)&\le \int_{
s_\dd+\dd}^t\E\Big(|f(X_r^{(\dd)})|\Big|\F_s\Big)\d r=\int_{
s_\dd+\dd}^t\E\Big(\E\Big(|f(X_r^{(\dd)})|\Big|\F_{s_\dd+\delta}\Big)\big|\F_s\Big)\d
r\\
&=\int_{ s_\dd+\dd}^t\E\Big(
\E|f(X_{s_\dd+\dd,r}^{(\dd),x})|\Big|_{x=X_{s_\dd+\dd}^{(\dd)}}
\Big|\F_s\Big)\d r.
\end{split}
\end{equation*}
Applying Theorem \ref{lem0} and H\"older's inequality, we conclude that
\begin{equation*}
\begin{split}
\E|f(X_{s_\dd+\dd,r}^{(\dd),x})|&=\int_{\R^d}|f(y)|
 p^{(\delta)}(s_\dd+\dd, x;r,y)\d y\\
&\le C\int_{\R^d}|f(y)|
 p_\alpha(r-s_\dd-\dd, y-x)\d y
 \le
C(r-s_\dd-\dd)^{-\ff{d}{\alpha q}}\|f\|_{q}.
\end{split}
\end{equation*}
Therefore, it holds that
\begin{equation}\label{H3}
\begin{split}
I_2(s,t)\leq C
(t-s)^{1-d/(\alpha q)}\|f\|_{q},
\end{split}
\end{equation}
which, together with \eqref{H2}, implies \eqref{OO}.

The remaining procedure of deriving \eqref{D1} from \eqref{OO} is standard. For readers' convenience, we sketch it here. For each $k\ge1$,
applying inductively \eqref{OO} gives
\begin{equation}\label{S3}
\begin{split}
\E\bigg(\bigg(\int_s^t|f(X^{(\dd)}_r)| \d r\Big)^k\bigg|\F_s\bigg)
&=k! \E\Big(
\int_{\triangle_{k-1}(s,t)}|f(X^{(\dd)}_{r_1})|\cdots|f(X^{(\dd)}_{r_{k-1}})|\d r_1\\
&\qquad\cdots\d r_{k-1}
\times\E\Big(\int_{r_{k-1}}^t|f(X^{(\dd)}_{r_k})|\d
r_k\Big|\F_{r_{k-1}}\Big)\Big|\F_s\Big)\\
&\le k!(C (t-s)^{1-d/(\alpha q)}\|f\|_{q}
)^k,~~~~0\le s\le t\le T,
\end{split}
\end{equation}
where
\begin{equation*}
\triangle_k(s,t):=\{(r_1,\cdots,r_{k})\in\R^{k}:s\le r_1\le\cdots\le
r_{k}\le t\}.
\end{equation*}
Taking $\dd_0=(2C\ll\|f\|_{q})^{-\frac{1}{1-d/(\alpha q)}}$, and combining this with \eqref{S3}, we derive
\begin{equation}\label{D7}
\E\Big(\exp\Big(\ll\int_{(i-1)\dd_0}^{i\dd_0\wedge
T}|f(X^{(\dd)}_t)| \d
t\Big)\Big|\F_{(i-1)\dd_0}\Big)\le\sum_{k=0}^\8\ff{1}{2^k}=2,~~~i\ge1,
\end{equation}
which further implies that
\begin{equation}\label{f2}
\begin{split}
\E\exp\bigg(\ll\int_0^T|f(X^{(\dd)}_t)| \d t\bigg)
&=\E\Big\{\E\bigg(\exp\bigg(\int_0^T|f(X^{(\dd)}_t)| \d
t\bigg)\Big|\sF_{\lfloor T/\dd_0\rfloor\dd_0}\Big\}\\
&=\E\bigg(\exp\bigg(\ll\sum_{i=1}^{\lfloor
T/\dd_0\rfloor}\int_{(i-1)\dd_0}^{ i\dd_0 }|f(X^{(\dd)}_t)| \d
t\bigg)\\
&\quad\times \E\bigg(\exp\bigg(\ll\int_{\lfloor
T/\dd_0\rfloor\dd_0}^{T }|f(X^{(\dd)}_t)| \d
t\bigg)\Big|\F_{\lfloor T/\dd_0\rfloor\dd_0}\bigg)\bigg)\\
&\le2\,\E\exp\bigg(\ll\sum_{i=1}^{\lfloor
T/\dd_0\rfloor}\int_{(i-1)\dd_0}^{ i\dd_0 }|f(X^{(\dd)}_t)| \d
t\bigg)\\
&\le\cdots\le 2^{1+T/\dd_0  }.
\end{split}
\end{equation}
Therefore, \eqref{D1} holds.
\end{proof}

The following lemma is concerned with Krylov's and Khasminskii's estimates for
the solution process $(X_t)_{t\ge0}$ to \eqref{1.1}, which is more or less
standard; see, for instance, \cite{GM,KR,XZ,Z,Z2}. Whereas, we
herein state them and provide a sketch of its proof by using explicit upper bound of heat kernel.
\begin{lem}\label{Kam}
Assume {\bf (A1)}. Then for any $q>(d/\alpha\vee1)$,
\begin{align}\label{kry}\E\Big(\int_s^t|f(X_r)|\d r\Big|\F_s\Big)\le
C\,\|f\|_{q} (t-s)^{1-d/(\alpha q)},~~~f\in L^q(\R^d), 0\le s\le t\le T,
\end{align}
and
\begin{equation}\label{f1}
\E\exp\Big(\ll\int_0^T|f(X_t)| \d t\Big)\le 2^{1+  T (\ll c\|f\|_{q})^{\frac{1}{1-d/(\alpha q)}}},\ \ f\in L^q(\R^d), \lambda>0
\end{equation}
hold for some constants $C,c>0$.
\end{lem}

\begin{proof}
By \cite[Theorem 1.5 and Remark 1.6]{CZ1}, we conclude that  the solution $X_t$  to SDE \eqref{1.1} has a probability density under condition {\bf(A1)}, and we denote it as $p(0,x;t,y)$. Moreover, it satisfies the following upper bounds,
	\begin{equation}\label{A0}
 	|p(0,x;t,y)|\le Cp_\alpha(t,y-x),~~~0< t\le T,
 		x,y\in\R^d.
 	\end{equation}
This, together with H\"older's inequality,
Markov property and \eqref{naao}, yields that
\begin{equation}\label{f3}
\begin{split}
\E\Big(\int_s^t|f(X_r)|\d
r\Big|\F_s\Big)&=\int_s^t\E\Big(|f(X_r^{s,x})|\Big|_{x=X_s}\Big)\d r\\
&\le C\int_s^t\int_{\R^d}|f(y)|p_\alpha(r-s,y-x)\d y\Big|_{x=X_s}\d r\\
&\le C(t-s)^{1-d/(\alpha q)}\|f\|_{q},
\end{split}
\end{equation}
where $(X^{s,x}_t)_{t\ge s}$ stands for the solution to \eqref{1.1}
with the initial value $X_s^{s,x}=x.$ By repeating the same procedure in the proof of \eqref{D1}, we can derive \eqref{f2}.
\end{proof}

 For a locally integrable function
$h:\R^d\to\R,$ the Hardy-Littlewood maximal operator $\mathscr{M}h$
is defined as below
\begin{equation*}
(\mathscr{M}h)(x)=\sup_{r>0}\ff{1}{|B_r(x)|}\int_{B_r(x)}h(y)\d
y,~~~~x\in\R^d,
\end{equation*}
where $B_r(x) $  is the ball with the radius $r$ centered at the
point $x$ and $|B_r(x)|$ denotes volume
of $B_r(x)$.
According to \cite[Lemma 5.4]{Z2}, the following Hardy-Littlewood
maximum theorem  holds.

\begin{lem}
There exists a constant $C>0$ such that for any continuous and weak differential function $f:\R^d\to\R$,
\begin{equation}\label{S1}
|f(x)-f(y)|\le C|x-y|\{(\mathscr{M}|\nn f|)(x)+(\mathscr{M}|\nn
f|)(y)\},~~~\mbox{ a.e. } x,y\in\R^d.
\end{equation}
Moreover, there exists a constant
$C_q>0$ such that for any $q>1$ and $f\in L^q(\R^d)$,
\begin{equation}\label{S2}
\|\mathscr{M}f\|_{q}\le C_q\|f\|_{q}.
\end{equation}
\end{lem}
We cite the following lemma  \cite[Lemma 2.3]{Z1} for future use.
\begin{lem}\label{du} Let $q>1$ and $\gamma\in[1,2]$. There exists a constant $C=C(q,\gamma,d)$ such that for any $f\in H_q^{\gamma}$,
$$\|f(\cdot+z)-f(\cdot)\|_{1,q}\leq |z|^{\gamma-1}\|f\|_{\gamma,q}.$$
\end{lem}
To overcome the difficulty caused by the singularity of drift $b$, we give some results on Zvonkin's transformation. More precisely,
for any $\lambda >0$, consider the following elliptic equation for
$u^\ll:\R^d\to\R^d$
\beq\label{PDE}
\beg{split}
\L u^\ll+b+\nabla_{b}u^\ll=\lambda
u^\ll,
\end{split}\end{equation}
where $b$ is given in \eqref{1.1}, $\sL$ is defined as
\beg{align}\label{oper}
\L f(x)=\int_{\R^d-\{\bf0\}}\{f(x+z)-f(x)-\<\nabla f(x),z\>1_{\{|z|\leq 1\}}\}\nu(\d z).
\end{align}
According to \cite[Theorem 4.11]{XZ} and Sobolev's embedding \eqref{embed}, we have the following lemma.
\begin{lem}\label{ua} Assume {\bf(A2)}. Then, for any  $\gamma\in((1+\alpha/2-\beta)\vee 1\vee (d/p-\beta+1),\alpha)$, there exists a constant $\lambda_0>0$ such that for any $\lambda\geq \lambda_0$ \eqref{PDE} has a
unique solution $u^{\lambda}\in H_p^{\gamma+\beta}$ satisfying
\begin{equation}\label{F3}
 \|\nn u^\ll\|_{\8}\le\frac{1}{2},\ \
 \|u^\ll\|_{{\gamma+\beta},p}\leq C_1\|b\|_{\beta,p},
\end{equation}
for some constant $C_1>0$.
\end{lem}

Now we are in position to complete the proof of Theorem \ref{th1}.
\begin{proof}[Proof of Theorem \ref{th1}]
Firstly, recall that the L\'evy-It\^o decomposition of $W_{S_t}$ is
\beg{align}\label{decomL}
W_{S_t}=\int_0^t\int_{\{|z|\le 1\}}z\til{N}(\d s,\d z)+\int_0^t\int_{\{|z|>1\}}zN(\d s,\d z),~~t\ge0,
\end{align}
here, $N$ is the Poisson random measure with compensator $\nu(\d z)\d t$.
 Set
$\theta^\ll(x):=x+u^\ll(x), x\in\R^d,$ and
$Z_t^{(\dd)}:=X_t-X^{(\dd)}_t$. According to It\^o's formula in \cite[Lemma 6.4]{XZ}, we obtain from
\eqref{PDE} that
\begin{equation*}
\begin{split}
\d\theta^\ll(X_t)&= \ll u^\ll(X_t)\d t+\d W_{S_t}+
\int_{\R^d- \{\bf0\}}[u^{\lambda}(X_{t-}+z)-u^\lambda(X_{t-})]\tilde{N}(\d t,\d z)\\
\d\theta^\ll(X^{(\dd)}_t)&=\ll u^\ll(X_t)\d t +\na \theta^{\la}(X_t^{(\de)})
( b(X^{(\dd)}_{t_\dd})-b(X^{(\dd)}_t))\d t\\
&+\d W_{S_t}+\int_{\R^d- \{\bf0\}}[u^\lambda(X^{(\delta)}_{t-}+z)-u^\lambda(X^{(\delta)}_{t-})]\tilde{N}(\d t, \d z).
\end{split}
\end{equation*}
Set $\bar{\theta^\la}(X_t,X_t^{(\de)}):=\theta^\ll(X_t)-\theta^\ll(X_t^{(\de)})$ and $g(x,z):=u^{\lambda}(x+z)-u^\lambda(x)$. Then,
it follows that
\beg{equation}\label{theta}
\begin{split}
\d\bar{\theta^\la}(X_t,X_t^{(\de)})
&=\ll(u^\ll(X_t)-u^\ll(X_t^{(\de)}))\d t+\na \theta^{\la}(X_t^{(\de)})(b(X^{(\dd)}_t)-b(X^{(\dd)}_{t_\dd}))\d t\\
&~~+\int_{|z|>1}
\left[g(X_{t-},z)-g(X^{(\delta)}_{t-},z)\right]\tilde{N}(\d t,\d z)\\
&~~+\int_{0<|z|\leq 1}\left[g(X_{t-},z)-g(X^{(\delta)}_{t-},z)
\right]\tilde{N}(\d t,\d z).
\end{split}
\end{equation}
 We obtain from \eqref{F3} that
 \beg{align}\label{upp-th}
 \ff{1}{4}|Z_t^{(\de)}|^2\le|\bar{\theta^\la}(X_t,X_t^{(\de)})|^2\le\ff{9}{4}|Z_t^{(\de)}|^2.
 \end{align}
 By It\^o's formula and \eqref{upp-th}, we arrive at
 \beg{align}\label{ZZ}
 |Z_t^{(\de)}|^2
 \nonumber&\leq 4|\bar{\theta^\la}(X_t,X_t^{(\de)})|^2\\ \nonumber
 &\le 8\la\int_0^t\<\bar{\theta^\la}(X_s,X_s^{(\de)}),u^\ll(X_s)-u^\ll(X_s^{(\de)})\>\d s\\\nonumber
 &+8\int_0^t\<\bar{\theta^\la}(X_s,X_s^{(\de)}),\na \theta^{\la}(X_s^{(\de)})(b(X_s^{(\de)})-b(X_{s_\de}^{(\de)}))\>\d s\\\nonumber
 &+4\int_0^t\int_{|z|\ge1}\Big|g(X_{s-},z)-g(X^{(\delta)}_{s-},z)\Big|^2\nu(\d z)\d s\\\nonumber
 &+4\int_0^t\int_{0<|z|\le1}\Big|g(X_{s-},z)-g(X^{(\delta)}_{s-},z)\Big|^2\nu(\d z)\d s\\
&+4\int_0^t\int_{|z|\ge1}\big\{|\hat{\theta^\la}(X_s,X_s^{(\de)},g)|^2-|\bar{\theta^\la}(X_s,X_s^{(\de)})|^2\big\}\til{N}(\d s,\d z)\\\nonumber
 &+4\int_0^t\int_{0<|z|\le1}\big\{|\hat{\theta^\la}(X_s,X_s^{(\de)},g)|^2-|\bar{\theta^\la}(X_s,X_s^{(\de)})|^2\big\}\til{N}(\d s,\d z)\\\nonumber
 &=:\sum_{i=1}^4I_i^{(\de)}(t)+M(t),
 \end{align}
where $\hat{\theta^\la}(X_t,X_t^{(\de)},g)=\bar{\theta^\la}((X_t,X_t^{(\de)})+g(X_{t-},z)-g(X^{(\delta)}_{t-},z))$.
 By means of  \eqref{F3} and \eqref{upp-th}, we obtain
 \beg{align}\label{I1}
 I_1^{(\de)}(t)\le6\la\int_0^t|Z_s^{(\de)}|^2\d s.
 \end{align}
 Similarly, by virtue of \eqref{F3}, \eqref{upp-th} and Young's inequality, we arrive at
 \beg{align}\label{I2}
  I_2^{(\de)}(t)\le C\Big\{\int_0^t|Z_s^{(\de)}|^2\d s+\int_0^t|b(X_s^{(\de)})-b(X_{s_\de}^{(\de)})|^2\d s\Big\}.
 \end{align}

Thanks to \cite[(2.9), (2.12)]{HL}, there exists a constant $ C(t,\nu) >0$ such that
\begin{equation}\label{I3}
\begin{split}
I_3^{(\de)}(t)&\le C\int_0^t\int_{|z|>1}\left|g(X_{u-},z)-g(X^{(\delta)}_{u-},z)\right|^2\d u\nu(\d z)\leq C(t,\nu)\int_0^t|Z_u^{(\dd)}|^2\d u.
\end{split}
\end{equation}
Let $\gamma\in ((1+\alpha/2-\beta)\vee 1\vee (d/p-\beta+1),\alpha)$.
Define
\beg{align*}U(x,z):=|z|^{1-\be-\ga}|\na u^\la(x+z)-\na u^\la(x)|.
\end{align*}
It follows that
\beg{align}\label{nag}
|\na g(\cdot,z)(x)|
&= U(x,z)|z|^{\be+\ga-1}.
\end{align}
Noting that
$(\scr M(f))^2(x)\leq \scr M(f^2)(x)$
due to Jensen's inequality. By \eqref{S1}, \eqref{nag} and \cite[(3.3)]{XZ} for $\mathbb{B}=L^2(\{0<|z|\leq 1\}, \nu)$, we get
\begin{equation*}
\begin{split}
&I_4^{(\de)}(t)\\
&=\int_0^t\int_{0<|z|\leq 1}\left|g(X_{s-},z)-g(X^{(\delta)}_{s-},z)\right|^2\d s\nu(\d z)\\
&\le 2C_2\int_0^t|Z_s^{(\delta)}|^2\Big\{\scr M\Big(\int_{0<|z|\leq 1}|\nabla g(\cdot,z)|^2\nu(\d z)\Big)(X_{s-})\\
&\qquad\qquad\qquad\qquad\qquad+\scr M\Big(\int_{0<|z|\leq 1}|\nabla g(\cdot,z)|^2\nu(\d z)\Big)(X^{(\delta)}_{s-})\Big\}\d s\\
&=2C_2\int_0^t|Z_s^{(\delta)}|^2\Big\{\scr M\Big(\int_{0<|z|\leq 1}|U(\cdot,z)|^2|z|^{2(\be+\ga-1)}\nu(\d z)\Big)(X_{s-})\\
&\qquad\qquad\qquad\qquad\qquad+\scr M\Big(\int_{0<|z|\leq 1}|U(\cdot,z)|^2|z|^{2(\be+\ga-1)}\nu(\d z)\Big)(X_{s-}^{(\de)})\Big\}\d s.
\end{split}
\end{equation*}
As a result, plugging this with \eqref{I1}-\eqref{I3} into \eqref{ZZ} gives that
\begin{align*}
\nonumber &|Z_t^{(\dd)}|^2\le C_3\int_0^t\sup_{r\in[0,s]}|Z_r^{(\dd)}|^2(\d
s+\d A_s)+\int_0^t C_3|b(X^{(\dd)}_s)-b(X^{(\dd)}_{s_\dd})|^2\d s+M_t,
\end{align*}
where $M_t$ is a local martingale, and \begin{align*}A_t&=\int_0^t\Big\{\scr M\Big(\int_{0<|z|\leq 1}|U(\cdot,z)|^2|z|^{2(\be+\ga-1)}\nu(\d z)\Big)(X_{s-})\\
&\qquad\qquad\qquad\qquad\qquad+\scr M\Big(\int_{0<|z|\leq 1}|U(\cdot,z)|^2|z|^{2(\be+\ga-1)}\nu(\d z)\Big)(X_{s-}^{(\de)})\Big\}\d s.
\end{align*}
By \eqref{S2} and Minkowski's inequality, we have
\begin{align*}\left\|\scr M\Big(\int_{0<|z|\leq 1}|U(\cdot,z)|^2|z|^{2(\be+\ga-1)}\nu(\d z)\Big)\right\|_{\frac{p}{2}}\leq C\int_{0<|z|\leq 1}\left\|U(\cdot,z)\right\|_p^2|z|^{2(\be+\ga-1)}\nu(\d z).
\end{align*}
This together with H\"older's inequality and the fact $\int_{0<|z|\leq 1}|z|^{2(\be+\ga-1)}\nu(\d z)<\infty$ due to  $2(\beta+\gamma-1)>\alpha$, we derive that for any $ \zeta>0$,
\beg{align*}
\E \exp\big\{\zeta A_t\}&\le\left(\E\exp\left\{2\zeta\int_0^t\scr M\Big(\int_{0<|z|\leq 1}|U(\cdot,z)|^2|z|^{2(\be+\ga-1)}\nu(\d z)\Big)(X_{s-})\d s\right\}\right)^{1/2}\\
&\times \left(\E\exp\left\{2\zeta\int_0^t\scr M\Big(\int_{0<|z|\leq 1}|U(\cdot,z)|^2|z|^{2(\be+\ga-1)}\nu(\d z)\Big)(X_{s-}^{(\delta)})\d s\right\}\right)^{1/2}\\
&\le 2^{\left\{1+t\left(2\zeta c\int_{0<|z|\leq 1}\left\|U(\cdot,z)\right\|_p^2|z|^{2(\be+\ga-1)}\nu(\d z)\right)^{\ff{\al p}{\al p-2d}}\right\}}\\
&\le 2^{\left\{1+t\left(2\zeta C\|b\|_{\beta,p}^2\right)^{\ff{\al p}{\al p-2d}}\right\}},
\end{align*}
where the second inequality is due to \eqref{D1}, \eqref{f1} for taking parameters $\lambda=2\zeta$ and $q=\ff{p}{2}$, and in the last display we used the fact that
\beg{align*}
\|U(\cdot,z)\|_p&=|z|^{1-\be-\ga}\Big(\int_{\R^d}|\na u^\la(x+z)-\na u^\la(x)|^p\d x\Big)^{1/p}\\\nonumber
&\le |z|^{1-\be-\ga}|z|^{\be+\ga-1}\|\na u\|_{\be+\ga-1,p}\le \|u\|_{\be+\ga,p}\leq  C\|b\|_{\be,p},
\end{align*}
which is due to Lemma \ref{ua} and Lemma \ref{du}.

Consequently, we deduce by stochastic Gronwall's inequality (see
e.g. \cite[Lemma 3.8]{XZ}) that, for $0<\kk'<\kappa<1$,
\begin{equation*}
\begin{split}
&\Big(\E\Big(\sup_{0\le s\le
t}|Z_s^{(\dd)}|^{2\kk'}\Big)\Big)^{1/\kk'}\\
&\le
\Big(\ff{\kk}{\kk-\kk'}\Big)^{1/\kk'}\Big(\E\e^{\kk
A_t/(1-\kk)}\Big)^{(1-\kk)/\kk}\times\int_0^t\Big\{C_3\E|b(X^{(\dd)}_s)-b(X^{(\dd)}_{s_\dd})|^2\Big\}\d s.
\end{split}
\end{equation*}
Taking $\kappa'=\ff{\eta}{2}$ and  combining with Lemma \ref{lem1} implies that \eqref{W1} holds.
\end{proof}

\beg{thebibliography}{99} {\small

\setlength{\baselineskip}{0.14in}
\parskip=0pt

\bibitem{BHY}Bao, J., Huang, X., Yuan, C., Convergence rate of Euler--Maruyama Scheme for SDEs with
H\"older-Dini continuous drifts,  {\it J. Theor. Probab.},   {\bf
	32} (2019),    848-871.

\bibitem{BHZ}Bao, J., Huang, X., Zhang, S., Convergence rate of EM algorithm for SDEs under integrability condition, {\it J. Appl. Probab.}, (2022), 1-24. doi:10.1017/jpr.2021.56

\bibitem{BG} Blumenthal, R.M., Getoor, R.K., Some theorems on stable processes, {\it Trans. Am. Math. Soc.}, {\bf
	95} (1960),  263-273.

\bibitem{CZ} Chen, Z, Zhang, X., Heat kernels and analyticity of non-symmetric jump diffusion semigroups,  {\it Probab. Theory Relat. Fields}, {\bf
	165} (2016),  267-312.
	
	\bibitem{CZ1} Chen, Z.,  Zhang, X., Heat kernels for time-dependent non-symmetric stable-like operators, {\it J. Math. Anal. Appl.},  {\bf 465} (2018), 1-21.

\bibitem{DKS} Dareiotis, K.,  Kumar, C.,  Sabanis, S., On
tamed Euler approximations of SDEs driven by L\'{e}vy noise with
applications to delay equations, {\it SIAM J. Numer. Anal.}, {\bf
	54} (2016),  1840-1872.

\bibitem{FGP}  Flandoli, M.,   Gubinelli, M.,   Priola,  E.,  Flow of diffeomorphisms
for SDEs with unbounded H\"{o}lder continuous drift, {\it Bull. Sci.
	Math.}, {\bf 134} (2010),   405-422.

\bibitem{GLN} Gottlich, S., Lux, K.,  Neuenkirch, A.,   The Euler scheme
for stochastic differential equations with discontinuous drift
coefficient: A numerical study of the convergence rate,
{\it Adv. Differ. Equ.}, {\bf  429} (2019), 1-21.

\bibitem{GM} Gy\"ongy, I.,   Martinez, T.,   On stochastic differential equations with locally unbounded drift,  {\it Czechoslovak Math.J.},
{\bf 51} (2001), 763-783.

\bibitem{HK}Halidias, N.,   Kloeden, P.~E., A note on the Euler-Maruyama
scheme for stochastic differential equations with a discontinuous
monotone drift coefficient, {\it BIT Numer. Math.},  {\bf48} (2008),  51-59.

\bibitem{HK1} Higham,D. J, Kloeden, P. E., Strong convergence rates for backward Euler on a class of nonlinear jump diffusion problems, {\it J. Comput. Appl. Math.}, {\bf 205} (2007), 949-956.

\bibitem{HL} Huang,  X.,  Liao, Z., The Euler-Maruyama method for S(F)DEs with H\"{o}lder drift and $\alpha$-stable noise, \textit{Stoch. Anal. Appl.}, {\bf 36} (2018), 28-39.

\bibitem{HW181} Huang, X., Wang, F.-Y., {Degenerate SDEs with singular drift and applications to Heisenberg groups,}  \textit{J. Differ. Equ.}, {\bf 265} (2018), 2745-2777.

\bibitem{KK}Knopova, V., Kulik, A., Parametrix construction of the transition probability density of the solution to an SDE driven by $\alpha$-stable noise, {\it Ann. Inst. H. Poincar\'e Probab. Statist. }, {\bf 54} (2018), 100-140.

\bibitem{KM}Konakov, V., Mammen, E., Local limits theorems for
transition densities of Markov chains converging to diffusions, {\it
	Probab. Theory Relat. Fields.}, {\bf 117} (2000), 551-587.
	
\bibitem{KM1}Konakov, V., Menozzi, S., Weak error for stable driven stochastic differential equations: Expansion of densities, {\it J. Theor. Probab.}, {\bf 24} (2011), 454-478.	
	
\bibitem{KP}Kloden, P.E., Platen, E., Numerical solutions of stochastic differential equations, {\it Springer, Berlin-Heidelberg.}	1992.

\bibitem{KR}   Krylov,  N.~V.,  R\"{o}ckner, M.,  Strong solutions of stochastic equations with singular time dependent drift,  {\it
	Probab. Theory Relat. Fields.}, {\bf 131} (2005), 154-196.

	\bibitem{KS}
	K\"{u}hn F., Schilling R.-L., Strong convergence of the Euler-Maruyama approximation for a class of L\'{e}vy-driven SDEs. {\it Stoch. Process  Their Appl.}, {\bf 129} (2019), 2654-2680.

\bibitem{LM} Lemaire, V., Menozzi, S., On some non asymptotic bounds
for the Euler scheme, {\it Electron. J. Probab.}, {\bf 15} (2010),
1645-1681.

\bibitem{LS}Leobacher, G.,   Sz\"olgyenyi, M.,  A numerical method for SDEs
with discontinuous drift, {\it BIT Numer. Math.}, {\bf 56} (2016), 151-162.

\bibitem{LS2} Leobacher, G.,   Sz\"olgyenyi, M.,   A strong order 1/2
method for multidimensional SDEs with discontinuous drift, {\it Ann.
	Appl. Probab.}, {\bf 27} (2017), 2383-2418.

\bibitem{LS3}Leobacher, G.,   Sz\"olgyenyi, M. , Convergence of the
Euler-Maruyama method for multidimensional SDEs with discontinuous
drift and degenerate diffusion coefficient, {\it Numer. Math.}, {\bf
	138}   (2018), 219-239.
	

\bibitem{PT}Pamen, O. M., Taguchi, D., Strong rate of convergence for the Euler-Maruyama approximation of SDE with H\"older continuous drift coefficient, {\it Stoch. Process Their Appl.}, {\bf 127} (2017), 2542-2559.

\bibitem{T1} Triebel, H., Interpolation theory, Functional Spaces, Differential Operators, \textit{North-Holland Publishing company}. 1978.

\bibitem{WW}Wang,. F.-Y., Wang, J., Harnack inequalities for stochastic equations driven by L\'evy noise, \textit{J. Math. Anal. Appl.} {\bf 410} (2014), 513-523.

\bibitem{XZ}  Xie, L., Zhang,  X.,   Ergodicity of stochastic differential equations with jumps and singular coefficients,    \textit{ Ann. Inst. H. Poincar\'e Probab. Statist.}, {\bf 56} (2020), 175-229.

\bibitem{Z} Zhang, X.,   Strong solutions of SDEs with singural drift and Sobolev diffusion coefficients,   {\it Stoch. Process Their Appl.}, {\bf 115} (2005),
1805-1818.

\bibitem{Z1} Zhang,  X., {Stochastic differential equations with Sobolev drifts and driven by $\alpha$-stable processes,} \textit{ Ann. Inst. H. Poincar\'e Probab. Statist. }, {\bf 49} (2013), 1057-1079.

\bibitem{Z2}Zhang,  X.,   Stochastic homeomorphism flows of SDEs with singular drifts and Sobolev diffusion coefficients,  {\it
	Electron. J. Probab.}, {\bf 16} (2011), 1096-1116.

\bibitem{AZ} Zvonkin, A.~K.,   A transformation of the phase space of a diffusion process that removes the drift,   {\it Math. USSR Sb.},  {\bf 93} (1974), 129-149.

}
\end{thebibliography}

\end{document}